\documentclass[11pt]{article}
\usepackage[calc]{picture}
\usepackage{graphicx}
\usepackage{subfigure}
\usepackage{graphics}
\usepackage{xcolor,tikz}
\usepackage{amssymb,amsmath,amsthm,bm,setspace,mathrsfs}
\usepackage{url,color,units}
\usepackage[T1]{fontenc}

\DeclareMathOperator{\Cov}{Cov}

\newcommand{\<}{\left\langle}
\renewcommand{\>}{\right\rangle}

\newcommand{\Z}{\mathbf{Z}}

\newcommand{\R}{\mathbf{R}}

\RequirePackage{color}
\RequirePackage[colorlinks, urlcolor=my-blue,linkcolor=my-link,citecolor=my-green]{hyperref}

\definecolor{my-link}{rgb}{0.5,0.0,0.0}
\definecolor{my-blue}{rgb}{0.0,0.0,0.6}
\definecolor{my-red}{rgb}{0.5,0.0,0.0}
\definecolor{my-green}{rgb}{0.0,0.5,0.0}
\definecolor{darkgreen}{rgb}{0.2,0.45,0}
\definecolor{nicos-red}{rgb}{0.65,0.0,0.0}
\definecolor{light-gray}{gray}{0.6}
\definecolor{really-light-gray}{gray}{0.8}

\newcommand{\be}{\begin{equation}}
\newcommand{\ee}{\end{equation}}

\newcommand{\lip}{\text{\rm Lip}}

\renewcommand{\P}{\mathrm{P}}
\newcommand{\E}{\mathrm{E}}
\newcommand{\F}{{\bf F}}

\renewcommand{\d}{{\rm d}}

\newcommand{\e}{{\rm e}}

\renewcommand{\ge}{\geqslant}
\renewcommand{\le}{\leqslant}

\author{Le Chen\\Univ.\ Utah
\and Michael Cranston\\U.C. Irvine
\and Davar Khoshnevisan\\Univ.\ Utah
\and Kunwoo Kim\\Univ.\ Utah}
\title{Dissipation and high disorder\thanks{
	Research supported in part by the Swiss Federal Fellowship  (L.C.)
	and the NSF grant DMS-1307470 (D.K.).}}
\date{Version: November 11, 2014}
\newtheorem{stat}{Statement}[section]

\newtheorem{theorem}[stat]{Theorem}
\newtheorem{lemma}[stat]{Lemma}
\theoremstyle{definition} 
\newtheorem*{definition}{Definition}
\newtheorem*{cremark}{Remark}

\numberwithin{equation}{section}

\begin{document}
\maketitle
\begin{abstract} 
 Given a field $\{B(x)\}_{x\in\Z^d}$ of independent standard Brownian motions, indexed by $\Z^d$, the generator of a suitable Markov process on $\Z^d,\,\,\mathcal{G},$ and sufficiently nice function $\sigma:[0,\infty)\to[0,\infty),$
we consider the influence of the parameter $\lambda$ on the behavior of the system,
\begin{eqnarray*}
	\d u_t(x) = & (\mathcal{G}u_t)(x)\,\d t +  \lambda\sigma(u_t(x))\d B_t(x)
	\qquad[t>0,\ x\in\Z^d],\\ 
	&u_0(x)=c_0\delta_0(x),
\end{eqnarray*}
 We show that for any $\lambda>0$ in dimensions one and two the total mass $\sum_{x\in\Z^d}u_t(x)\to 0$ as $t\to\infty$ while for dimensions greater than two there is a phase transition point $\lambda_c\in(0,\infty)$ such that for $\lambda>\lambda_c,\,  \sum_{\Z^d}u_t(x)\to 0$ as $t\to\infty$ while for $\lambda<\lambda_c,\,\sum_{\Z^d}u_t(x)\not\to 0$ as $t\to\infty.$ 


\end{abstract}

\section{Introduction}

Let $\tau$ denote a probability density function on $\Z^d$, and consider the
linear operator $\mathcal{G}$ defined by
\begin{equation}\label{G}
	(\mathcal{G}h)(x) = \sum_{y\in\Z^d}[h(x+y)-h(x)]\tau(y),
\end{equation}
for all $x\in\Z^d$ and bounded functions $h:\Z^d\to\R$. We may think of
$\mathcal{G}$ as the generator of a rate-one
continuous-time random walk $\bm{X}:=\{X_t\}_{t\ge0}$ on $\Z^d$%
---more correctly put, a compound Poisson process $X$---such that
$X_0=0$ and  $\tau(x)=\P\{X_J=x\,, J<\infty\}$ for all $x\in\Z^d$,
where $J$ denotes the first time the process $X$ jumps. In order to rule out
trivialities, we will assume that $X$ is genuinely $d$-dimensional. In particular,
$J<\infty$ a.s.\ and $\tau(x)=\P\{X_J=x\}$.

Let $\{B(x)\}_{x\in\Z^d}$
denote a field of independent standard Brownian motions, indexed by $\Z^d$,
and consider the system of It\^o stochastic ODEs,
\begin{equation}\label{SHE}
	\d u_t(x) =  (\mathcal{G}u_t)(x)\,\d t +  \lambda\sigma(u_t(x))\d B_t(x)
	\qquad[t>0,\ x\in\Z^d],
\end{equation}
subject to $u_0(x):= c_0\delta_0(x)$ for all $x\in\Z^d$, where $c_0,\lambda>0$
are finite and non random numerical quantities. We will think
of the number $c_0$ as fixed, and of $\lambda$ as a tuning parameter which
describes the \emph{level of the noise}.

Here and throughout we assume that $\sigma:\R\to\R$ is a deterministic
Lipschitz-continuous function. It follows from the work of Shiga and Shimizu
\cite{ShigaShimizu} that the particle system \eqref{SHE} has a unique strong solution.

We plan to study the solution to \eqref{SHE} under further mild restrictions on
the operator $\mathcal{G}$ and the nonlinearity $\sigma$.
Regarding $\mathcal{G}$, we always assume that $\tau$ has 
mean zero and compact support; the latter is equivalent to 
the notion that $\mathcal{G}$ is finite range. To summarize, we have
\begin{equation}\label{EM}
	\sum_{x\in\Z^d} x_j\tau(x)=0\quad\text{for all $1\le j\le d$}
	\quad\text{and}\quad
	\max_{\|x\|>R_0}\tau(x)=0
\end{equation}
for some $R_0\in(1\,,\infty)$.
In order to rule out trivialities, we assume also that $\tau(0)<1$.
Otherwise, \eqref{SHE} describes a countable family of independent
and/or noninteracting one-dimensional It\^o diffusions. It is easy to see
that the best-studied example of \eqref{SHE} is included here; that is
the case where $\mathcal{G}$ is the discrete Laplacian,
$(\mathcal{G}h)(x)=(2d)^{-1}\sum_{y\in\Z^d:\, |y-x|=1} h(y)$
where $|z|:=\sum_{j=1}^d|z_j|$ for all $z\in\Z^d$.
Other examples abound.

As regards the nonlinearity, we will alway assume that 
\begin{equation}\label{sig}
	\sigma(0)=0\quad\text{and}\quad
	{\rm L}_\sigma:=\inf_{z\in\R}|\sigma(z)/z|>0.
\end{equation}
The first part of this condition ensures that the solution $u$ to \eqref{SHE}
is ``physical.'' More precisely, the strict inequality
$u_t(x)> 0$ holds for all $t> 0$ and $x\in\Z^d$ almost surely; see Georgiou
et al \cite[Lemma 7.1]{GJKS}.
The second is a ``moment intermittency condition'' \cite{FK,Shiga1992}.

The \emph{parabolic Anderson model} $\sigma(u)= u$ has been studied a great deal
(see Carmona and Molchanov \cite{CarmonaMolchanov})
in part because it arises
naturally in other disciplines, and also because it is close to being
an exactly-solvable model. In fact, in a few cases, it is exactly solvable;
see Borodin and Corwin \cite{BorodinCorwin}.

Thanks to a comparison argument \cite[Theorem 5.1]{GJKS},
Theorem 1.2 of Shiga \cite{Shiga1992} implies that there exists a number 
$\lambda_1>0$ such that 
\begin{equation}
	\lim_{t\to\infty} u_t(x)=0\qquad\text{a.s.\ for all $x\in\Z^d$},
\end{equation}
if $\lambda>\lambda_1$. 
One can recast this, somewhat informally, as the assertion that the solution to
\eqref{SHE} is locally dissipative under \emph{strong disorder}; see
Carmona and Hu \cite{CU} for the terminology on strong vs.\ weak disorder.

On the other hand, the theory of
Georgiou et al \cite{GJKS} implies that if $d\ge 3$, then there exists a finite
and positive number $\lambda_2$ such that
\begin{equation}
	\lim_{t\to\infty} \sup_{x\in\Z^d}u_t(x)=0\qquad\text{ a.s.,}
\end{equation}
whenever $\lambda\in(0\,,\lambda_2)$.
This implies that the solution to
\eqref{SHE} is uniformly---hence also locally---dissipative under \emph{weak disorder}.

Finally, let us mention that when there is no disorder, that is when $\sigma\equiv 0$, the
solution to the Kolmogorov--Fokker--Planck equation
\eqref{SHE} is simply $u_t(x) = \P\{X_t=-x\}$, which goes to zero uniformly
in $x$ as $t\to\infty$
thanks to a suitable form of the local central limit theorem.

Thus, we see that
{\it local} dissipation is a generic property of \eqref{SHE}, regardless
of the strength of the disorder in \eqref{SHE}.
By contrast, the main result of this  paper shows that
{\it global} dissipation essentially characterizes the presence of strong disorder.
In order to describe our result, consider the \emph{total mass process}
\[
	m_t(\lambda) := \|u_t\|_{\ell^1(\Z^d)} := \sum_{x\in\Z^d}
	|u_t(x)|\qquad[t\ge0].
\]
It is well known that $t\mapsto m_t(\lambda)$ is a mean-$c_0$
continuous $L^2(\P)$-martingale. As far as we know,
a variation on this observation goes on one hand at least as far back as Spitzer's paper 
\cite[Proposition 2.3]{Spitzer1981}
on discrete [more-or-less linear] interacting particle systems. More closely-related
variations can be found in the
literature on measure-valued diffusions (see  Dawson and Perkins \cite{DawsonPerkins}
for pointers to the literature).
The particular case
that we need follows from \eqref{mild} below and the fact that
$m_t(\lambda)>0$ for all $t>0$, a.s.
The asserted positivity follows from
Lemma 7.1 of Georgiou et al \cite{GJKS} which implies that
\begin{equation}\label{eq:positivity:u}
	u_t(x)>0\qquad\text{for all $x\in\Z^d,\,t>0$, almost surely.}
\end{equation}
Owing to the martingale convergence
theorem, one consequence of positivity is that
\begin{equation}\label{m_infty}
	m_\infty(\lambda) := \lim_{t\to\infty} m_t(\lambda)
\end{equation}
exists a.s.\ and is finite a.s. for all $\lambda>0$.

\begin{definition}
	We say that \eqref{SHE} is \emph{globally dissipative} if
	$m_\infty(\lambda)=0$ a.s.
\end{definition}

Frequently, the probability literature refers to this property as
``extinction.'' We prefer ``dissipation'' because a correct interpretation
of ``extinction,'' in the present setting, might suggest the false
claim that $m_t(\lambda)=0$ a.s.\
for all $t$ sufficiently large, since as mentioned above, the strict inequality
$u_t(x)> 0$ holds for all $t> 0$ and $x\in\Z^d$ almost surely. 


The principle result of this paper is the following, which essentially
equates global dissipation with the presence of strong disorder.

\begin{theorem}\label{th:main}
	In recurrent dimensions $[d=1,2]$, the system \eqref{SHE}
	is always globally dissipative. In transient dimensions
	$[d\ge 3]$ there is a sharp phase transition; namely,
	there exists a nonrandom number $\lambda_c\in(0\,,\infty)$
	such that \eqref{SHE} is globally dissipative if $\lambda>\lambda_c$
	and not globally dissipative if $\lambda\in(0\,,\lambda_c)$.
\end{theorem}

\begin{cremark}
	The case $\lambda=\lambda_c$ is open.
\end{cremark}

Theorem \ref{th:main} is a qualitative result, but its proof has
some quantitative aspects as well. In particular, as part of
the proof, we will demonstrate that there exists a finite random
variable $V:=V(\lambda,\sigma,c_0,d)$ and a nonrandom
positive and finite constant $v=v(\lambda,\sigma,c_0,d)$ such that
\begin{equation}\label{eq:quant}
	m_t(\lambda) \le V\times\begin{cases}
		\exp\left( -v t^{1/3} \right)&\text{if $d=1$},\\
		\exp\left( -v\sqrt{\log t}\right)&\text{if $d=2$},
	\end{cases}
\end{equation} 
almost surely for all $t>1$. We do not know if these bounds are sharp, only that
\begin{equation}\label{eq:quant:LB}
	\limsup_{t\to\infty} \frac1t \log m_t(\lambda) >-\infty,
\end{equation}
with positive probability in all dimensions $d\ge 1$ and for all $\lambda>0$.
However, our methods are in some sense robust:
We will prove that some aspects of  \eqref{eq:quant} can be carried out in the
continuous setting of stochastic partial differential equations as well
[see Theorem \ref{th:decay:SHE}].

Let us conclude the Introduction with a few remarks about the literature.

In the case that $d=1,2$, the qualititative part of
Theorem \ref{th:main} is a part of the folklore of the subject,
and follows from well-known
ideas about linear interacting particle systems---see for example
Liggett \cite[Theorem 4.5, p.\ 451]{Liggett} and
especially Shiga \cite[Remark 4]{Shiga1992}. 

Shiga \cite[p.\ 793]{Shiga1992} asserts that ``it is plausible
that the extinction occurs'' when $d\ge 3$. The transient-dimension portion
of Theorem \ref{th:main} disproves Shiga's prediction
when the noise level is sufficiently low. In the language of interacting particle systems, 
Theorem \ref{th:main} implies the ``survival'' of the solution to \eqref{SHE}
in transient dimensions when $\lambda$ is small. Our
method of proof of system survival is quite different from the more familiar
ergodic-theoretic ones and worthy of attention in its own right; 
e.g., compare with Liggett \cite[Chapter IX, \S2]{Liggett}.

Throughout, $\lip_\sigma$ denotes the optimal Lipschitz constant of
$\sigma$; that is,
\begin{equation}
	\lip_\sigma := \sup_{-\infty<x<y<\infty} \left| \frac{\sigma(x)-\sigma(y)}{
	x-y}\right|.
\end{equation}
Of course, $\lip_\sigma<\infty$ by default.

\section{Some technical estimates}
In this section we record three elementary technical facts that we will
soon need. One [Lemma \ref{lem:tails}] is a
variation on very well-known large deviations estimates for L\'evy processes.
The other two  [Lemmas \ref{lem:delta<2} and Lemma \ref{lem:delta=2}]
contain extremal bounds on subsolutions to a certain infinite family of
differential equations.

Let $Y_1,Y_2,\ldots$ be i.i.d.\ random variables in $\Z^d$
such that $\P\{Y_1=x\}=\tau(x)$ for all $x\in\Z^d$. In particular,
$Y_1$ has mean zero and  moment generating function
\begin{equation}
	\varphi(z):=\E\exp(z\cdot Y_1),
\end{equation}
that is finite for all $z$ in an open neighborhood of the origin of $\R^d$.

Let $N:=\{N(t)\}_{t\ge 0}$ denote an independent rate-one
Poisson process, and consider the compound Poisson process
[sometimes also called continuous-time random walk],
\begin{equation}
	X_t := \sum_{j=1}^{N(t)} Y_j\qquad[t\ge0],
\end{equation}
where $\sum_{j=1}^0Y_j :=0$. Clearly, $\{X_t\}_{t\ge0}$ is a L\'evy process
on $\Z^d$ whose generator $\mathcal{G}$ is defined in \eqref{G}.

\begin{lemma}\label{lem:tails}
	Under the preceding conditions, for every $q\in(0\,,\infty)$
	there exists  $c\in(0\,,\infty)$ such that
	\begin{equation}\label{LD}
		\P\left\{\|X_t\|>K\right\} \le 2d\exp\left(-cK^2/t\right),
	\end{equation}
	uniformly for all $t\ge1$ and $K\in[0\,,qt]$.
\end{lemma}


Lemma \ref{lem:tails} is basically a version of Hoeffding's inequality \cite{Hoeffding}
in continuous time, and can be obtained from Hoeffding's inequality
by first conditioning on $N(t)$.
Next we describe the second, more analytic, portion of this section.

Choose and fix $\alpha,\delta,\gamma>0$, and define 
$\F(\alpha\,,\delta\,,\gamma)$ to be the collection of
all non-negative continuously-differentiable functions $f:\R_+\to\R_+$
such that
\begin{equation}\label{eq:DE}
	f'(t) \le -\alpha
	\sup_{K\in[a,bt]}\left[ \frac{f(t) - \exp\left(-\gamma K^2/t\right)}{K^\delta}\right]
	\qquad\text{for all $t\ge1$},
\end{equation}
and some $0<a<b$. We will reserve the notation
$\F(\alpha\,,\delta\,,\gamma)$ as this function class throughout the
paper.

Suppose $f\in \F(\alpha\,,\delta\,,\gamma)$ for some finite
numbers $\alpha,\delta,\gamma>0$. Because $f(t)\ge 0$ 
for all $t>0$, we can set $K:=bt$ in the optimization problem that defines
$\F(\alpha\,,\delta\,,\gamma)$ in order to conclude that
\begin{equation}
	f'(t) \le \alpha(bt)^{-\delta}\exp(-\gamma b^2 t)
	\qquad\text{for all $t\ge 1$.}
\end{equation}
Consequently, $f$ is bounded. The following
gives a strong improvement in the case that $\delta<2$.

\begin{lemma}\label{lem:delta<2}
	For every $\delta\in[0\,,2)$, $\alpha,\gamma>0$, and
	$f\in\F(\alpha\,,\delta\,,\gamma)$,
	\begin{equation}
		\limsup_{t\to\infty} 
		\frac{\log f(t)}{t^\nu}<0,\quad
		\text{with $\nu:=\frac{2-\delta}{2+\delta}$}.
	\end{equation}
\end{lemma}

\begin{proof}
	Define $\beta:=4/(2+\delta)$ and observe that 
	$\beta\in(1\,,2]$ since $\delta\in[0\,,2)$. We appeal to \eqref{eq:DE}
	with $K:= t^{\beta/2}$ in order to see that every 
	$f\in\F(\alpha\,,\delta\,,\gamma)$
	satisfies
	\begin{equation}\label{eq:DE1}
		f'(t) +  \alpha t^{-2\delta/(2+\delta)}f(t) \le \alpha \exp\left( -\gamma 
		t^\nu\right),
	\end{equation}
	uniformly for all $t$ sufficiently large. Define
	\begin{equation}
		g(t) := \exp\left( \theta t^\nu\right) f(t)\qquad[t\ge 0],
	\end{equation}
	where $\theta$ is a fixed parameter that satisfies
	\begin{equation}
		0 < \theta < \min\left( \gamma\,, \frac\alpha\nu\right).
	\end{equation}
	Then, \eqref{eq:DE1} ensures that $g$ satisfies
	\begin{equation}\begin{split}
		g'(t) &= \exp\left( \theta t^\nu\right) \left[
			f'(t) + \theta\nu t^{-2\delta/(2+\delta)} f(t)\right]\\
		&<  \alpha\exp\left( -[\gamma-\theta] t^\nu\right),
	\end{split}\end{equation}
	for all $t$ sufficiently large. This implies that $g$ is bounded,
	which is another way to state the lemma.
\end{proof}

The preceding proof works also when $\delta=2$, and shows that  in that case
every function $f\in \F(\alpha\,,2\,,\gamma)$ is bounded
for every $\alpha,\gamma>0$. But this is vacuous, as we have seen already. 

Next we study the case that $\delta=2$ more carefully and show among other
things that 
if $f\in\F(\alpha\,,2\,,\gamma)$ for some $\alpha,\gamma>0$, then $f(t)$ tends to $0$
faster than any negative power of $\log t$ as $t\to\infty$.

\begin{lemma}\label{lem:delta=2}
	For every $\alpha,\gamma>0$ and $f\in\F(\alpha\,,2\,,\gamma)$,
	\begin{equation}
		\limsup_{t\to\infty} \frac{\log f(t)}{\left(\log t\right)^{1/2}} <0.
	\end{equation}
\end{lemma}

\begin{proof}
	The argument is similar to the proof of Lemma \ref{lem:delta<2},
	but we need to make a few modifications. Specifically, we
	now use $K:= t^{1/2}(\log t)^{1/4}$, and 
	$g(t) :=\exp\{\theta\sqrt{\log t}\}f(t)$ for a sufficiently small
	constant $\theta>0$. The remaining details are routine and left to 
	the interested reader.
\end{proof}

\section{Proof of Theorem \ref{th:main}}

The proof is split into separate parts. First, let us
define $\{p_t\}_{t\ge0}$ to be the transition functions
of the underlying walk $X$; that is,
\begin{equation}
	p_t (x) := \P\{X_t=x\}\qquad\text{for all $t\ge0$ and $x\in\Z^d$}.
\end{equation}
These functions play a role in our analysis, since the solution
$u$ to \eqref{SHE} can be written in the following integral form:
\begin{equation}\label{mild}
	u_t(x) = c_0p_t(-x) + \sum_{y\in\Z^d}\int_0^t p_{t-s}(y-x)
	\sigma(u_s(y))\,\d B_s(y);
\end{equation}
see Shiga and Shimizu \cite{ShigaShimizu}. Now we proceed with the proof,
which is split into a number of distinct steps.

\subsection{Proof in recurrent dimensions}
	We begin by proving \eqref{eq:quant};
	Theorem \ref{th:main} follows immediately in recurrent
	dimensions; that is, when $d\in\{1\,,2\}$.
	
	The proof in recurrent dimensions proceeds by estimating
	fractional moments of $m_t(\lambda)$; see Chapter XII
	of Liggett \cite{Liggett} for similar ideas in the context of discrete
	particle systems and Mueller and Tribe \cite{MuellerTribe} in the
	context of continuous systems.
	
	As was mentioned in the Introduction, 
	it is well known that $\{m_t(\lambda)\}_{t\ge0}$ is a continuous 
	$L^2(\P)$-martingale
	with $\E[m_t(\lambda)]=c_0$ for all $t\ge0$. This is obtained by summing \eqref{mild}
	over $x\in\Z^d$ on both  sides in order to see that
	\begin{equation}\label{M:int}
		m_t(\lambda) = c_0 + \lambda\sum_{y\in\Z^d}\int_0^t \sigma(u_s(y))\, \d B_s(y)
		\qquad[t\ge0].
	\end{equation}
	Because $\sigma(0)=0$ [see \eqref{sig}], it follows that 
	\begin{equation}\label{lipsig}
		|\sigma(z)|\le \lip_\sigma|z|\qquad
		\text{for all $z\in\R$.}
	\end{equation}
	Therefore,
	the exchange of summation and stochastic integration is a standard consequence
	of measurability and the fact that 
	\begin{equation}\label{M:M}
		\sum_{y\in\Z^d}\int_0^t
		\E(|\sigma(u_s(y))|^2)\,\d s\le \lip_\sigma^2\sum_{y\in\Z^d}\int_0^t
		\E(|u_s(y)|^2)\,\d s<\infty,
	\end{equation}
	for all $t>0$. See (2.14) of Shiga and Shimizu \cite{ShigaShimizu}
	for a qualitative statement and Lemma 8.1 of
	Georgiou et al \cite[Lemma 8.1]{GJKS} for a
	sharp quantitative version. 
		
	Because of \eqref{M:M} and the It\^o isometry, 
	$m(\lambda):=\{m_t(\lambda)\}_{t\ge0}$ is also 
	an $L^2(\P)$-martingale,
	and the quadratic variation process of $m(\lambda)$ is described by
	\begin{equation}\label{<m(lambda)>}
		\langle m(\lambda) \rangle_t = 
		\lambda^2\sum_{y\in\Z^d}\int_0^t |\sigma(u_s(y))|^2\, \d s
		=\lambda^2\int_0^t \|\sigma\circ u_s\|_{\ell^2(\Z^d)}^2\,\d s.
	\end{equation}
	Therefore, \eqref{sig} and \eqref{lipsig} together yield the 
	a.s.\ inequalities,
	\begin{equation}\label{<m(lambda)>:bis}
		\lambda^2{\rm L}_\sigma^2 \int_0^t \|u_s\|_{\ell^2(\Z^d)}^2\d s\le
		\langle m(\lambda) \rangle_t  \le 
		\lambda^2 \lip_\sigma^2\int_0^t\|u_s\|_{\ell^2(\Z^d)}^2\d s,
	\end{equation}
	valid for all $t>0$.
	
	Since $m_t(\lambda)\ge u_t(0)$, Eq.\ (7.2) of Georgiou
	et al \cite{GJKS} guarantees that for every $T>0$ there exists $C_T:=C_T(\lambda)
	\in(0\,,\infty)$
	such that
	\begin{equation}
		\P\left\{  \inf_{t\in[0,T]}m_t(\lambda) <\varepsilon \right\} \le C_T 
		\varepsilon^{\log\log (1/\varepsilon)/C_T}\quad\text{for all
		$\varepsilon\in(0\,,1)$}.
	\end{equation}
	This shows in particular that 
	$\sup_{t\in[0,T]} | m_t(\lambda)|\in L^p(\P)$
	for all $p\in (-\infty,\infty),\\T\in(0\,,\infty)$.
	Consequently, we may  apply It\^o's formula
	to see that for all $\eta\in(0\,,1)$,
	\begin{align}
		&[m_t(\lambda)]^\eta\\\notag
		&=c_0^\eta + \eta\int_0^t [m_s(\lambda)]^{\eta-1}\,\d m_s(\lambda) -
			\frac{\lambda^2\eta(1-\eta)}{2}\int_0^t [m_s(\lambda)]^\eta \frac{\|\sigma\circ
			u_s\|_{\ell^2(\Z^d)}^2}{\|
			u_s\|_{\ell^1(\Z^d)}^2}\,\d s,
	\end{align}
	almost surely, where the stochastic integrals are bona fide continuous
	$L^2(\P)$-martingales. In particular,
	\begin{equation}\label{df}
		\E\left([m_t(\lambda)]^\eta\right) = c_0^\eta - 
		\frac{\lambda^2\eta(1-\eta)}{2}\int_0^t \E\left([m_s(\lambda)]^\eta \frac{\|\sigma\circ
		u_s\|_{\ell^2(\Z^d)}^2}{\|
		u_s\|_{\ell^1(\Z^d)}^2}\right)\d s,
	\end{equation}
	for every $t>0$ and $\eta\in(0\,,1)$. The preceding is true also for
	$\eta\ge1$, but we care only about values of $\eta$ in $(0\,,1)$.
	
	Because of \eqref{sig},
	$\|\sigma\circ u_s\|_{\ell^2(\Z^d)}\ge
	{\rm L}_\sigma\|u_s\|_{\ell^2(\Z^d)}.$
	Therefore, the nonrandom function
	$t\mapsto\E([m_t(\lambda)]^\eta)$ is continuously
	differentiable and solves
	\begin{equation}\label{df3}
		f'(t) \le  - \frac{\lambda^2\eta(1-\eta){\rm L}_\sigma^2}{2} \E\left(
		[m_t(\lambda)]^\eta R_t^2\right)
		\qquad\text{for all $t>0$},
	\end{equation}
	where
	\begin{equation}
		R_t := \frac{\|u_t\|_{\ell^2(\Z^d)}}{\|
		u_t\|_{\ell^1(\Z^d)}}.
	\end{equation}
	
	For every real number $K\ge 1$, let 
	\begin{equation}\label{eq:B(K)}
		\mathcal{B}(K):=\{x\in\Z^d:\, \|x\|\le K\}.
	\end{equation}
	There exists a positive and finite constant $c:=c(d)$
	such that the cardinality of $\mathcal{B}(K)$ is at least $c^{-1} K^d$, uniformly for
	all $K\ge 1$. Therefore,
	by the Cauchy--Schwarz inequality, 
	\begin{equation}\begin{split}
		\|u_t\|_{\ell^2(\Z^d)}^2 &\ge \sum_{x\in\mathcal{B}(K)} [u_t(x)]^2\\
		&\ge cK^{-d} \left(\sum_{x\in\mathcal{B}(K)}u_t(x)\right)^2\\
		&= cK^{-d}\left( \|u_t\|_{\ell^1(\Z^d)} - \sum_{x\not\in\mathcal{B}(K)}
			 u_t(x)\right)^2.
	\end{split}\end{equation}
	for every $t,K>0$.  Consequently,
	\begin{equation}\label{lowerR}\begin{split}
		R_t^2 &\ge cK^{-d}\left( 1 -  \frac{\sum_{x\not\in\mathcal{B}(K)}u_t(x)}{
			\sum_{x\in\Z^d} u_t(x)}\right)^2\\
		&\ge cK^{-d}\left( 1 - 2 \frac{\sum_{x\not\in\mathcal{B}(K)}u_t(x)}{
			\sum_{x\in\Z^d} u_t(x)}\right),
	\end{split}\end{equation}
	and hence \eqref{df3} implies that
	\begin{equation}\begin{split}
		f'(t) &\le - \frac{c\lambda^2\eta(1-\eta){\rm L}_\sigma^2}{2K^d}\left( f(t) - 2\E\left[
			\frac{\sum_{x\not\in\mathcal{B}(K)}u_t(x)}{%
			\left(\sum_{x\in\Z^d} u_t(x)\right)^{1-\eta}}\right]\right)\\
		&\le - \frac{c\lambda^2\eta(1-\eta){\rm L}_\sigma^2}{2K^d}\left( f(t) - 2\E\left[
			\left(\sum_{x\not\in\mathcal{B}(K)}u_t(x)\right)^{\eta}\right]\right);
	\end{split}\end{equation}
	the last line holds
	merely because 
	\begin{equation}
		\left\{\sum_{x\not\in\mathcal{B}(K)}u_t(x)\right\}^{1-\eta}\le
		\left\{\sum_{x\in\Z^d}u_t(x) \right\}^{1-\eta}.
	\end{equation}
	By Jensen's inequality and Lemma \ref{lem:tails}, 
	we can find $c\in(0\,,\infty)$ such that
	\begin{equation}\begin{split}
		\E\left[\left(\sum_{x\not\in\mathcal{B}(K)}u_t(x)\right)^{\eta}\right]
		 	&\le\left(\E\left[\sum_{x\not\in\mathcal{B}(K)}u_t(x)\right]\right)^\eta\\
		&= \left( \sum_{x\not\in\mathcal{B}(K)} c_0p_t(x)\right)^\eta\\
		&\le (2c_0d)^\eta \exp\left(-c\eta K^2/t\right),
	\end{split}\end{equation}
	uniformly for all $K\in[1\,,t]$. Therefore,
	\begin{equation*}
		f'(t) \le - \frac{c\lambda^2\eta(1-\eta){\rm L}_\sigma^2}{2}\sup_{K\in[1,t]}
			\left( \frac{f(t) - 2(2c_0d)^\eta
			\exp\left(-c\eta K^2/t\right)}{K^d}\right)
			\end{equation*}
			and so with $C=(2(2c_0d)^\eta)^{-1},$ we have		
	\begin{equation}
		Cf'(t)\le - \frac{ c\lambda^2
			\eta(1-\eta){\rm L}_\sigma^2}{2}\sup_{K\in[1,t]}
			\left( \frac{Cf(t) -  \exp\left(-c\eta K^2/t\right)}{K^d}\right),
	\end{equation}
	uniformly for all $t\ge1$. In other words, $Cf$ is an element of 
	$\F(\alpha\,,d\,,c\eta)$, where 
	$\alpha:= c\lambda^2\eta(1-\eta){\rm L}_\sigma^2/2$.
	Because of this fact, we may employ Lemmas \ref{lem:delta<2}
	and \ref{lem:delta=2} in order to deduce the existence of
	constants $V:=V(\eta\,,\lambda)\in(1\,,\infty),\,v:=v(\eta\,,\lambda)\in(0\,,\infty)$ such that for all $t\ge1$,
	\begin{equation}\label{cases}
		\E\left([m_t(\lambda)]^\eta\right) \le V\times\begin{cases}
			\exp\left(-vt^{1/3}\right)&\text{if $d=1$},\\
			\exp\left( -v\sqrt{\log t}\right)&\text{if $d=2$}.
		\end{cases}
	\end{equation}
	
	If $U_1,\ldots,U_n$ is a nonnegative supermartingale, then
	Doob's inequality tells us that
	$\lambda\P\{\max_{1\le j\le n}U_j>\lambda\}\le\E(U_1)$ for
	all $\lambda>0$. Since
	$\{[m_s(\lambda)]^\eta\}_{s\ge t}$ is a continuous nonnegative supermartingale
	for every fixed $t>0$,
	Doob's inequality and a standard approximation argument together yield
	\begin{equation}
		\P\left\{ \sup_{s\ge t} m_s(\lambda) > a\right\}
		\le a^{-\eta} \E\left([m_t(\lambda)]^\eta\right),
	\end{equation}
	for all $t,a>0$ and $\eta\in(0\,,1)$. 
	When $d=1$, this and \eqref{cases} together imply that
	\begin{equation}\begin{split}
		P_n &:= \P\left\{ \sup_{s\ge n-1}m_s(\lambda) > 
			\exp\left( -v\,n^{1/3}\right)\right\}\\
		&\le V\exp\left( v(\eta n^{1/3}-(n-1)^{1/3})\right),
	\end{split}\end{equation}
	for all integers $n\ge 1$. Since $\sum_{n-1}^\infty P_n<\infty$, the Borel--Cantelli lemma
	implies the existence of an integer-valued random variable $n_0$ such that
	\begin{equation}
		\sup_{s\ge n-1}m_s(\lambda) \le
		\exp\left( -v\,n^{1/3}\right)\qquad\text{for all $n> n_0$ a.s.}
	\end{equation}
	If $t> n_0$ is an arbitrary number, random or otherwise, 
	then we can find a unique integer $n\ge n_0$ such that
	$n-1 \le t\le n$. Then clearly,
	\begin{equation}
		m_t(\lambda) \le \sup_{s\ge n-1}m_s(\lambda)\le
		\exp\left( -v\,n^{1/3}\right)\le
		\exp\left( -v\,t^{1/3}\right)\qquad\text{a.s.}
	\end{equation}
	This inequality yields the first bound in \eqref{eq:quant},
	whence Theorem \ref{th:main} when $d=1$. 
	
	The proof of part 2 of \eqref{eq:quant}
	is essentially the same as the proof of part 1, 
	but when $d=2$ we use the second estimate in
	\eqref{cases} instead of the first one there. This proves Theorem \ref{th:main}
	for $d=2$.\qed

\subsection{Proof in transient dimensions: Existence of a unique phase transition}

In the second step of the proof we show the existence of a unique phase transition.
In principle, the proof is valid regardless of the value of the ambient dimension.
However, it will turn out that the phase transition is nontrivial only when $d\ge 3$.

Let us write the solution to \eqref{SHE} as $u_t(x\,;\lambda)$,
in order to emphasize the dependence of the solution on the size $\lambda$
of the underlying noise. Recall that $\lambda>0$ is a free parameter. Therefore,
the preceding
constructs $\lambda\mapsto u_\bullet(\bullet\,;\lambda)$ as a coupling
of stochastic processes, as well.

According to a comparison theorem of Cox et al
\cite{CFG}, for all integers $N\ge 1$, and all real 
$t>0$, 
\begin{equation}
	\E\exp\left( -\sum_{x\in	\mathcal{B}(N)} u_t(x\,;\bar\lambda)\right) \ge 
	\E\exp\left( -\sum_{x\in \mathcal{B}(N)} u_t(x\,;\lambda)\right),
\end{equation}
as long as $\bar\lambda\ge\lambda>0$. We let $N\uparrow\infty$ ,
and appeal to the monotone convergence theorem, in order to see that
$\E\exp ( -m_t(\bar\lambda) ) \ge \E\exp ( -m_t(\lambda) )$
for all $t>0$,
as long as $\bar\lambda\ge\lambda>0$. Now let $t\to\infty$
in order to deduce from the dominated convergence theorem that
\begin{equation}
	\E\exp ( -m_\infty(\bar\lambda) ) \ge
	\E\exp ( -m_\infty(\lambda) ),
\end{equation}
as long as $\bar\lambda\ge\lambda>0$. In other words,
$\lambda\mapsto\E\exp(-m_\infty(\lambda))$ is nondecreasing.
Thus,
\begin{equation}\label{infsup}
	\lambda_c := \sup\left\{ \lambda>0:\ \E\e^{ -m_\infty(\lambda)}<1
	\right\}=\inf\left\{ \lambda>0:\ \E\e^{ -m_\infty(\lambda)}
	=1\right\},
\end{equation}
where $\inf\varnothing:=+\infty$ and $\sup\varnothing:=0$. By the monotone
convergence theorem,
\begin{equation}
	\lambda_c = \sup\left\{ \lambda>0:\ m_\infty(\lambda)>0
	\text{ with positive probability}\right\},
\end{equation}
from the ``sup'' formula in \eqref{infsup}; also,
\begin{equation}
	\lambda_c = \inf\left\{ \lambda>0:\ m_\infty(\lambda)=0
	\text{ a.s.}\right\},
\end{equation}
from the ``inf'' formula in \eqref{infsup}.
This proves the existence of a unique $\lambda_c\in[0\,,\infty]$
with the properties mentioned in Theorem \ref{th:main}. 
The already-verified portion of the proof 
implies that $\lambda_c=0$ when $d=1,2$.
In the next two parts of the proof will show the nontriviality of $\lambda_c$
in transient dimensions; namely,
that $0<\lambda_c<\infty$
when $d\ge 3$. This endeavor will conclude the proof.
\qed

\subsection{Proof in transient dimensions: Supercritical phase}
In this section we consider only dimensions $d\ge 3$,
and demonstrate that $m_\infty(\lambda)=0$ a.s.
if $\lambda$ is sufficiently large. This immediately proves that 
\begin{equation}\label{lambda_c:finite}
	\lambda_c<\infty.
\end{equation}

We follow carefully Shiga's proof of his Theorem 1.2 
\cite[pp.\ 800--806]{Shiga1992}, keeping track of
the various sums and estimating them by elementary
means in order to find that  for all $\lambda$ sufficiently large
[$\kappa$ small, in the notation of Shiga]
there exists a constant
$c\in(0\,,\infty)$ such that
\begin{equation}\label{asL:P}
	\sup_{x\in\Z^d}
	\P\left\{ u_t(x) > \e^{-t/c} \right\} \le c\e^{-t/c}\qquad\text{for
	all $t\ge 1$}.
\end{equation}

Among other things, this readily implies the following weak formulation
of a ``local extinction result'': 
\begin{equation}\label{asL}
	\liminf_{t\to\infty} \frac{\log u_t(x)}{t} <0\quad\text{a.s.\ for all $x\in\Z^d$}.
\end{equation}
This is another way to say that the ``almost-sure Lyapunov exponent 
of the solution is negative.'' 
When $\sigma(x)=x$ and $\mathcal{G}$ is the discrete Laplacian---that is, when $\tau$ is uniform
distribution on the graph neighbors of the origin in $\Z^d$---\eqref{asL} is known
to hold with a limit in place of a $\liminf$; see Carmona and Molchanov
\cite{CarmonaMolchanov}. The most  complete results, in this case, can
be found in
Carmona et al \cite{CarmonaKoralevMolchanov} and Cranston et al \cite{CMS2002}.
More generally still,
Shiga \cite{Shiga1992}  considered the same class of nonlinear functions $\sigma$
as we do, and established \eqref{asL} with a proper limit in place of a
liminf.

We now suppose that $\lambda$ is large enough to ensure
the validity of \eqref{asL:P}, and
derive \eqref{lambda_c:finite} as follows. Recall $\mathcal{B}(K)$
from \eqref{eq:B(K)} and let $|\mathcal{B}(K)|$ denote its cardinality.
Setting $A(t)=\{\max_{x\in \mathcal{B}(t^2)} u_t(x) > 
	\frac{\eta}{|\mathcal{B}(t^2)|}\},$ we have by Chebyshev's inequality, 
\begin{eqnarray}
\begin{split}
	\P\left\{ m_t(\lambda) > 2\eta\right\}
	= &\P\left\{ \{m_t(\lambda) > 2\eta\}\cap A(t)\right\}
	+ \P\left\{ \{m_t(\lambda) > 2\eta\}\cap A(t)^c\right\}\\
\le& \P\left\{ A(t)\right\}+\P\left\{ \sum_{x\not\in\mathcal{B}(t^2)}u_t(x)>\eta)\right\}\\
\le&\P\left\{ A(t)\right\}+\eta^{-1}\sum_{x\not\in\mathcal{B}(t^2)}E[u_t(x)],
\end{split}
\end{eqnarray}
for all $\eta>0$ and sufficiently large $t>1$.
Since $|\mathcal{B}(t^2)|\sim\text{const}\cdot t^{2d}$
as $t\to\infty$,
Shiga's estimate \eqref{asL:P} ensures that
\begin{equation}
	\P\left\{ \max_{x\in\mathcal{B}(t^2)} u_t(x) > 
	\frac{\eta}{|\mathcal{B}(t^2)|}\right\} =O(t^{2d})\e^{-t/c}\qquad
	\text{as $t\to\infty$},
\end{equation}
whereas \eqref{mild} and Lemma \ref{lem:tails} together ensure that there
exist finite and positive constants $c_1$ and $c_2$ such that
\begin{equation}
	\sum_{x\not\in\mathcal{B}(t^2)}\E\left[ u_t(x)\right]
	= c_0\P\{ \|X_t\|> t^2\}
	\le c_1 \e^{-c_2t},
\end{equation}
for all $t>1$ sufficiently large. Thus,
that $m_t(\lambda)\to 0$ in probability as $t\to\infty$, and hence
$m_\infty(\lambda)=0$ a.s.\
for all $\lambda$ sufficiently large; \eqref{lambda_c:finite} follows.
\qed

\subsection{Proof in transient dimensions: Subcritical phase}
We continue to assume that $d\ge 3$, and
now prove that  $\lambda_c>0.$

Let  $\{X_t'\}_{t\ge0}$ be an independent copy of the continuous-time
random walk $X$ whose generator, we recall, is $\mathcal{G}$,
and define
\begin{equation}
	\Upsilon(0) := \int_0^\infty \P\{X_t=X_t'\}\,\d t.
\end{equation}
This is the total expected local time of the symmetrized walk $X-X'$
at the origin of $\Z^d$. It is well known that $\Upsilon(0)$ is finite because
$X-X'$ is a $d$-dimensional non-trivial random walk and hence
transient; see Chung and Fuchs \cite{ChungFuchs}. In fact, if $r$ is the probability of return to the origin for $X-X'$ then $\Upsilon(0)$ has an exponential distribution with parameter $2(1-r).$

Choose and fix any $\lambda>0$ that satisfies
\begin{equation}\label{l:small}	
	\lambda < \left[\lip_\sigma\sqrt{\Upsilon(0)}\right]^{-1}.
\end{equation}
According to Proposition 8.3 of Georgiou et al \cite{GJKS},
\begin{equation}\label{eq:GJKS:1}
	\sup_{t\ge0}
	\E\left( \left| m_t(\lambda)\right|^2\right) \le
	2c_0^2 \left( \frac{1+\varepsilon}{1-\varepsilon}\right),
\end{equation}
where $0<\varepsilon:=\lambda^2\lip_\sigma^2\Upsilon(0)<1$.
The Paley--Zygmund inequality is the following form
of the Cauchy--Schwarz inequality:
\begin{equation}
	\P \{ W\ge c_0/2\} \ge \frac{c_0^2}{4\E(W^2)},
\end{equation}
valid for every nonnegative
mean-$c_0$ random variable $W\in L^2(\P)$.
We choose
$W:=m_t(\lambda)$ to see that
\begin{equation}
	\delta:=\inf_{t\ge0}\P \{ m_t(\lambda) \ge c_0/2\}>0,
\end{equation}
as long as $\lambda$ satisfies \eqref{l:small}.
Thus, $\P\{m_\infty(\lambda)\ge c_0/2\}\ge\delta>0$ for all
such values of $\lambda$, and hence
$\lambda_c \ge  [\lip_\sigma\sqrt{\Upsilon(0)} ]^{-1}>0$,
as desired.\qed

\subsection{Proof of \eqref{eq:quant:LB}}
We conclude this section by establishing the 
quantitative lower bound \eqref{eq:quant:LB} that is valid in
all dimensions. Throughout this discussion, $\lambda>0$ is
held fixed.

Notice that
\begin{equation}
	\|\sigma\circ u_s\|_{\ell^2(\Z^d)}\le
	\lip_\sigma^2\|u_s\|_{\ell^2(\Z^d)}\le\lip_\sigma^2\|u_s\|_{\ell^1(\Z^d)},
\end{equation}
almost surely for all $s>0$.
Therefore, we may apply \eqref{df} to see that 
\begin{equation}\label{df:LB}
	\frac{\d}{\d t}\E\left([m_t(\lambda)]^\eta\right) 
	\ge -\frac{\lambda^2\eta(1-\eta)\lip_\sigma^2}{2}\E\left([m_t(\lambda)]^\eta\right),
\end{equation}
for all $t>0$ and $\eta\in(0\,,1)$. Choose and fix some $\eta\in(0\,,1)$
in order to see that the function
\begin{equation}
	f(t) := \E ([m_t(\lambda)]^\eta )\qquad[t>0]
\end{equation}
solves the differential inequality,
\begin{equation}
	f'(t) \ge -\frac12\lambda^2\eta(1-\eta)\lip_\sigma^2 f(t),
\end{equation}
for $t>0$, subject to $f(0)=1$. Therefore,
\begin{equation}
	\E\left([m_t(\lambda)]^\eta\right) 
	\ge \exp\left(-\frac{\lambda^2\eta(1-\eta)\lip_\sigma^2}{2}\,t\right),
\end{equation}
for all $t>0$ and $\eta\in(0\,,1)$. We apply the preceding with $t>1$ and
$\eta:=\eta_t:=1/t$ in order to see that
\begin{equation}\begin{split}
	\exp\left( -\frac{\lambda^2\lip^2_\sigma}{2}\right)
		&\le \exp\left(-\frac{\lambda^2\eta_t(1-\eta_t)\lip_\sigma^2}{2}\,t\right)\\
	&\le \E\left([m_t(\lambda)]^{\eta_t}\right)\\
	&\le \E\left([m_t(\lambda)]^{\eta_t};\ m_t(\lambda)\ge\e^{-ct}\right) 
		+ \e^{-c},
\end{split}\end{equation}
for all $c>0$. We apply the preceding with
an arbitrary choice of 
\begin{equation}\label{c}
	c>\tfrac12\lambda^2\lip^2_\sigma.
\end{equation}
Since $\eta_t\in(0\,,1)$, H\"older's inequality yields
\begin{equation}\begin{split}
	\E\left([m_t(\lambda)]^{\eta_t};\ m_t(\lambda)\ge\e^{-ct}\right) 
		&\le \left[\E\left(m_t(\lambda)\right)\right]^{1/t}
		\left[ \P\{m_t(\lambda)\ge\e^{-ct}\}\right]^{(t-1)/t}\\
	&= c_0^{1/t}\left[ \P\left\{m_t(\lambda)\ge\e^{-ct}\right\}\right]^{(t-1)/t}.
\end{split}\end{equation}
In this way we find that, as long as $c$ satisfies \eqref{c},
\begin{equation}\begin{split}
	\P\left\{m_t(\lambda)\ge\e^{-ct}\right\}
		&\ge c_0^{1/(1-t)}\left[\exp\left( -\frac{\lambda^2\lip^2_\sigma}{2}\right)
		- \e^{-c}\right]^{t/(t-1)}\\
	&\to\exp\left( -\frac{\lambda^2\lip^2_\sigma}{2}\right)
		- \e^{-c}>0,
\end{split}\end{equation}
as $t\to\infty$.
This implies  \eqref{eq:quant:LB}.\qed

\section{The stochastic heat equation on the real line}

We conclude this paper by showing how one can adjust our 
methods in order to study continuous
stochastic partial differential equations
[SPDEs]. Indeed, let $\xi:=\{\xi_t(x)\}_{t>0,x\in\R}$ denote a space-time
white noise; that is, a centered generalized Gaussian noise with
covariance measure,
\begin{equation}
	\Cov\left[\xi_t(x)\,,\xi_s(y)\right] = \delta_0(t-s)
	\delta_0(x-y)\qquad(s,t>0,\, x,y\in\R).
\end{equation}
We consider the SPDE,
\begin{equation}\label{cSHE}
	\dot{\psi}_t(x) = \tfrac12 \psi''_t(x) + \sigma(\psi_t(x))\xi_t(x),
\end{equation}
valid for all $t>0$ and $x\in\R$, subject to a non-random initial profile
$\psi_0\in L^\infty(\R),$ with $\psi_0\ge0.$ The nonlinearity $\sigma$
is, as before, a deterministic Lipschitz-continuous function that satisfies
\eqref{sig}. 

It is well known \cite[Ch.\ 3]{Walsh} that the SPDE \eqref{cSHE} 
has a unique continuous [weak]  solution $\psi$ that satisfies
\begin{equation}\label{MOM}
	\sup_{t\in[0,T]}\sup_{x\in\R} \E\left(|\psi_t(x)|^k\right)<\infty,
\end{equation}
for all $T>0$ and $k\ge 1$. That solution $\psi$ is also known to have the
following integral formulation \cite[Ch.\ 3]{Walsh},
\begin{equation}\label{cmild}
	\psi_t(x) = (G_t*\psi_0)(x) + \int_{(0,t)\times\R} G_{t-s}(y-x)
	\sigma(\psi_s(y))\,\xi(\d s\,\d y),
\end{equation}
where the stochastic integral is a Walsh integral \cite[Ch.\ 2]{Walsh}
and $G$ denotes the heat kernel,
\begin{equation}
	G_t(x) := \frac{1}{\sqrt{2\pi t}}\exp\left( - \frac{x^2}{2t}\right)
	\qquad(t>0,\,x\in\R).
\end{equation}
Then we have the following.

\begin{theorem}\label{th:decay:SHE}
	Suppose in addition that: (i) $\limsup_{|x|\to\infty}
	x^{-2}\log\psi_0(x)<0$
	and (ii) $\|\psi_0\|_{L^1(\R)}>0$. Then,
	$\psi_t\in L^1(\R)$ a.s.\ for all $t>0$, and
	\begin{equation}
		\limsup_{t\to\infty} \frac{1}{t^{1/3}}\log \| \psi_t\|_{L^1(\R)} <0
		\qquad\text{a.s.}
	\end{equation}
\end{theorem}

\begin{proof}
	By Mueller's comparison theorem \cite{Mueller1,Mueller2},
	$\psi_t(x)\ge0$ for all $t\ge0$ and $x\in\R$ off a single null set.
	Therefore, 
	\begin{equation}
		\mathcal{M}_t := \|\psi_t\|_{L^1(\R)}=\int_{-\infty}^\infty\psi_t(x)\,\d x.
	\end{equation}
	
	An \emph{a priori} estimate, similar to those in Dalang and Mueller
	\cite{DalangMueller}, can be used to show that since $\sigma(0)=0$
	and $\psi_0\in L^1(\R)$, $\psi_t\in L^1(\R)$ a.s.\ for all $t>0$.
	Moreover, we can integrate both sides of \eqref{mild} $[\d x]$
	in order to see that $t\mapsto\mathcal{M}_t$ a.s.\ solves
	the following for all $t>0$:
	\begin{equation}\label{Mint}
		\mathcal{M}_t  =\mathcal{M}_0  +
		\int_{(0,t)\times\R}\sigma(\psi_s(y))\,\xi(\d s\,\d y),
	\end{equation}
	with $\mathcal{M}_0:=\|\psi_0\|_{L^1(\R)}$.
	The exchange of the Lebesgue integral and the stochastic integral is
	justified by an appeal to
	a stochastic Fubini theorem \cite[Theorem 2.6, p.\ 296]{Walsh}.
	
	The identity \eqref{Mint} is the continuous analogue of 
	\eqref{M:int}, and shows that, parallel to the discrete setting,
	the total mass $\mathcal{M}$ is a non-negative,
	continuous $L^2(\Omega)$-martingale
	with mean $\mathcal{M}_0$ and quadratic variation,
	\begin{equation}
		\< \mathcal{M}\>_t = \int_0^t\d s\int_{-\infty}^\infty
		|\sigma(\psi_s(y))|^2\d y.
	\end{equation}
	In particular, \eqref{sig} and the Lipschitz continuity of $\sigma$ together
	yield the following: For all $t>0$,
	\begin{equation}
		{\rm L}_\sigma^2\int_0^t\|\psi_s\|_{L^2(\R)}^2\,\d s\le
		\<\mathcal{M}\>_t \le \lip_\sigma^2\int_0^t\|\psi_s\|_{L^2(\R)}^2\,\d s
		\qquad\text{a.s.}
	\end{equation}
	By It\^o's formula, if $\eta\in(0\,,1)$ is non random, then almost surely
	for all $t>0$,
	\begin{equation}
		\mathcal{M}_t^\eta
		= \mathcal{M}_0^\eta + \eta\int_0^t
		\mathcal{M}_s^{\eta-1}\,\d \mathcal{M}_s + 
		\frac{\eta(\eta-1)}{2}
		\int_0^t\mathcal{M}_s^{\eta-2}\d\<\mathcal{M}\>_s.
	\end{equation}
	The appeal to It\^o's formula, and the fact that the preceding stochastic integral
	is a bona fide martingale, both follow immediately from the fact that
	$\E(\sup_{s\in[0,t]}\psi_s^{-\mu})<\infty$ for all $t>0$ and $\mu>0$;
	see Mueller and Nualart \cite{MN}. 
	
	We integrate both sides of
	the preceding display $[\d\P]$---in a similar vein as
	was done for \eqref{df} and \eqref{df3}---in order to obtain the following:
	\begin{equation}
		\frac{\d}{\d t}\E\left(\mathcal{M}_t^\eta\right) 
		\le- \frac{\eta(1-\eta){\rm L}^2_\sigma}{2}\,\E\left( 
		\mathcal{M}_t^\eta\cdot\mathcal{R}_t\right),
	\end{equation}
	where
	\begin{equation}
		\mathcal{R}_s := \frac{\|\psi_s\|_{L^2(\R)}^2}{\|\psi_s\|_{L^1(\R)}^2}
		\qquad(s>0).
	\end{equation}
	Since $\psi_s$ has finite [negative and positive]
 moments of all orders, $\mathcal{R}_s$ does too. 
	Now we choose and fix an arbitrary nonrandom constant $K>0$,
	and argue as in \eqref{lowerR} to see that
	\begin{equation}
		\mathcal{R}_s
		\ge \frac{1}{2K}\left( 1 - \frac{2}{\|\psi_s\|_{L^1(\R)}}
		\int_{|x|>K}\psi_s(x)\,\d x\right),
	\end{equation}
	for all $s\ge0$. In particular,
	\begin{align}\notag
		\frac{\d}{\d t}\E\left(\mathcal{M}_t^\eta\right) &\le 
			- \frac{\eta(1-\eta){\rm L}^2_\sigma}{4K} \,\E\left( 
			\mathcal{M}_t^\eta\right)\\
		&\qquad + \frac{\eta(1-\eta){\rm L}_\sigma^2}{2K}\,
			\E\left( \|\psi_t\|_{L^2(\R)}^{\eta-1}
			\cdot\int_{|x|>K}\psi_t(x)\,\d x\right)
			\label{choochoo}\\\notag
		&\le -\frac{\eta(1-\eta){\rm L}^2_\sigma}{4K}\,\E\left( 
			\mathcal{M}_t^\eta\right)
			+ \frac{\eta(1-\eta){\rm L}_\sigma^2}{2K}\,
			\E\left( \left[\int_{|x|>K}\psi_t(x)\,\d x\right]^\eta\right),
	\end{align}
	since $[\int_{|x|>K}\psi_t(x)\,\d x/\|\psi_t\|_{L^1(\R)}]^{1-\eta}\le 1.$
	In order to estimate the last quantity in the preceding display,
	we appeal to Jensen's inequality:
	\begin{equation}\begin{split}
		\E\left( \left[\int_{|x|>K}\psi_t(x)\,\d x\right]^\eta\right)
			&\le \left[\E\int_{|x|>K}\psi_t(x)\,\d x\right]^\eta\\
		&=\left[ \int_{|x|>K} (G_t*\psi_0)(x)\,\d x\right]^\eta;
	\end{split}\end{equation}
	valid since $\E[\psi_t(x)]= (G_t*\psi_0)(x)$ by \eqref{mild}.
	Now a few lines of elementary calculations show that
	because $\psi_0$ decays as a Gaussian function, we
	can find finite and positive constants $c_1$ and $c_2$---independently
	of $t$---such that
	\begin{equation}
		\E\left( \left[\int_{|x|>K}\psi_t(x)\,\d x\right]^\eta\right)
		\le c_1\e^{-c_2 K^2/t}.
	\end{equation}
	Because of \eqref{choochoo}, this proves that
	\begin{equation}
		f(t) := \E\left( \|\psi_t\|_{L^1(\R)}^\eta\right)
	\end{equation}
	satisfies the pointwise inequality,
	\begin{equation}
		f'(t) \le -\frac{\eta(1-\eta){\rm L}^2_\sigma}{4K} f(t) + 
		\frac{c_1\eta(1-\eta){\rm L}_\sigma^2}{2K}\exp\left(-\frac{K^2}{c_2 t}\right).
	\end{equation}
	Consequently, there exist finite and positive constants $C$, $\alpha$, and
	$\gamma$ such that $f\in\F(\alpha\,,1\,,\gamma)$,
	whence $\log f(t) \le -C t^{1/3}$ for all $t\gg1$, thanks to Lemma
	\ref{lem:delta<2}, and hence that
	\begin{equation}
		\E\left(\|\psi_t\|_{L^1(\R)}^\eta\right) 
		\le C_1\exp\left(-\frac{t^{1/3}}{C_1}\right)
		\qquad\text{for all $t\ge C_1$}.
	\end{equation}
	Since $t\mapsto\|\psi_t\|_{L^1(\R)}^\eta$ is a nonnegative supermartingale---%
	see \eqref{Mint}---we apply Doob's inequality and a Borel--Cantelli argument
	to finish the proof.
\end{proof}

\spacing{.8}\begin{small}

\bigskip

\noindent {\bf Le Chen.} Department of Mathematics, University of Utah,
	Salt Lake City, UT 84112-0090,
	\textcolor{purple}{\textcolor{purple}{\texttt{chenle02@gmail.com}}}\\

\noindent {\bf Michael Cranston.} Department of Mathematics, 
	University of California-Irvine, Irvine, CA 92697-3875,
	\textcolor{purple}{\texttt{mcransto@gmail.com}}\\

\noindent {\bf D. Khoshnevisan.} Department of Mathematics, University of Utah,
	Salt Lake City, UT 84112-0090,
	\textcolor{purple}{\textcolor{purple}{\texttt{davar@math.utah.edu}}}\\

\noindent {\bf Kunwoo Kim.} Department of Mathematics, University of Utah,
	Salt Lake City, UT 84112-0090,
	\textcolor{purple}{\texttt{kkim@math.utah.edu}}
\end{small}

\end{document}